\documentclass[12pt, dvipdfmx]{amsart}

\usepackage{mathptmx}
\usepackage{amsmath,amssymb,amsthm}
\usepackage{graphics}
\usepackage{enumerate}
\usepackage[dvips]{color}
\usepackage{version}
\usepackage{stmaryrd}
\usepackage[all]{xy}
\usepackage{tikz}

\usepackage{todonotes}

\newtheorem{Thm}{Theorem}[section] %%%%%%%%%%%%%%%%%%%%%

\newtheorem{Cor}[Thm]{Corollary}
\newtheorem{Cor.Conj}[Thm]{Corollary of Conjecture}

\newtheorem{Prop}[Thm]{Proposition}

\newtheorem{Conj}[Thm]{Conjecture}

\newtheorem{Ques}[Thm]{Question}

\newtheorem{Claim}[Thm]{Claim}

\theoremstyle{remark}
\newtheorem{Rem}[Thm]{Remark}
\newtheorem{Ex}[Thm]{Example}

\theoremstyle{definition}
\newtheorem{Def}[Thm]{Definition}

\newtheorem{Step}{Step}

\newcommand{\R}{\ensuremath{\mathbb{R}}}

\newcommand{\Q}{\mathbb{Q}}

\usepackage{tikz}
\usetikzlibrary{shapes.geometric}
\usetikzlibrary {shapes.misc}
\usetikzlibrary{positioning}

\begin{document}
\nocite{GW}

\title[Semitoric vs log minimal models vs weak K-moduli]{Semi-toric and toroidal compactifications \\ 
as log minimal models, and applications to \\ weak K-moduli}
\author{Yuji Odaka}
\date{\today}

\maketitle

\begin{abstract}
We give a characterization of toroidal (resp., semi-toric) compactifications due to \cite{AMRT} (resp., \cite{Loo1, Loo2}) 
as log minimal models and apply it to study weak K-moduli compactifications, giving a different proof to a theorem of Alexeev-Engel. 
We also discuss towards further generalization, in particular revisit Shah-Sterk compactification of 
moduli of Enriques surfaces to show compatibility with log K-stability. 
\end{abstract}

%\tableofcontents

\section{Introduction}

We discuss compactifications of 
moduli of polarized K-trivial varieties $(X,L)$, 
which admit a description as locally symmetric spaces due to Torelli type theorems. 
This includes the case when $X$s are 
abelian varieties, K3 surfaces or more generally hyperK\"ahler varieties. 
In particular, we are motivated by semi-classical question to compare two kinds of compactifications; 
\begin{enumerate}
\item \label{cptf.i}
on one hand, a certain kind of compactification whose boundary parametrizes degenerate varieties, 
which we called 
``weak K-moduli", or its normalizations if necessary, 
\item \label{cptf.ii}
toroidal compactification (\cite{AMRT}) or semitoric compactifications (\cite{Loo1, Loo2, Loo3, Loo4}) on the other hand. 
\end{enumerate}

Recall that the former {\it weak K-moduli} means the moduli algebraic stacks of 
polarized semi-log-canonical (slc, for short) K-trivial varieties and their 
coarse moduli (cf., \cite{logmin}). ``K-" of K-moduli comes from K-(semi)stability, hence ``K"\"ahler originally, rather than canonical divisor $K_X$ (\cite{Od10}). 
Indeed, by \cite{Od, Od2}, such polarized varieties are characterized by K-semistability, 
if the canonical divisor is numerically trivial. 

One recent breakthrough for K3 surfaces in this direction is achieved in the paper of Alexeev-Engel \cite{AE} as a general statement (also cf., \cite{GHKS}), 
and its related papers \cite{ABE, AET, AE2, AEH} show explicit examples. Here, we attempt to 
improve and generalize parts of {\it op.cit}. 

\vspace{8mm}

On the way, as one of our main theorems (Theorem \ref{Loo.bir}), we also prove a birational characterization of toroidal compactification 
and semitoric compactification of Looijenga as {\it log minimal models} in the sense of the recent (log) minimal model program. 
The fact that it is log minimal is essentially semi-classical \cite{Mum77} while our point is to show the other direction. 

Then, we apply it to prove a variant of \cite[Main theorem 1.2]{AE} which also depends on our previous work \cite{logmin}. 
Note that {\it op.cit} depends on a variant of cone theorem \cite{Sva}, as well as admissible variation of mixed Hodge structures. 

\vspace{4mm}

To start with the original inspiration, we review the achievement of \cite{AE}. 
They discuss the $19$-dimensional quasi-projective coarse moduli 
$$F_{2d}=\{(X,L)\mid  (L^2)=2d\}/\sim,$$
where $d\in \mathbb{Z}_{>0}$, $X$s are K3 surfaces possibly with ADE singularities, 
$L$s are their primitive ample line bundles of degree $2d$, 
$\sim$ is the natural equivalence relation as isomorphisms of $X$ preserving $L$. 
Recall the fundamental result that it is biholomorphic to 
a Hermitian locally symmetric space of orthogonal type as follows. 
For the K3 lattice $\Lambda_{\rm K3}:=U^{\oplus  3}\oplus E_8(-1)^{\oplus 2}$ and its element $\lambda_{2d}:=e+df$ of norm $2d$, 
we set $\Lambda_{2d}:=\lambda_{2d}^{\perp}\subset \Lambda_{K3}$. Then 
we write  
$$\Omega(\Lambda_{2d}):=\{[w]\in \mathbb{P}(\Lambda_{2d}\otimes \mathbb{C})\mid
 (w,w)=0,\ (w,\bar{w})>0\},$$
 let 
$O(\Lambda_{\rm K3})$ be the orthogonal 
group of the lattice $\Lambda$ i.e., automorphisms preserving the bilinear form, 
and denote its stable subgroup as 
$$\tilde{O}(\Lambda_{2d}):=\{g|_{\Lambda_{2d}} : g\in O(\Lambda_{\rm K3}),\, 
g(\lambda)=\lambda\}.$$ Then, there is an isomorphism due to \cite{PS.Sha, Tod.surj}; 
$$
F_{2d}
\simeq
\tilde{O}(\Lambda_{2d})\backslash \Omega(\Lambda_{2d}). 
$$
This identification allows us to apply various compactification theories of \eqref{cptf.ii} side. 
Also, by using local Weyl group quotients, the above description can be enhanced to a description of the corresponding moduli stack as a 
Deligne-Mumford stack $\mathcal{F}_{2d}$ which underlies the universal family 
$$\pi_{2d}\colon (\mathcal{U}_{2d},\mathcal{L}_{2d})\twoheadrightarrow \mathcal{F}_{2d}.$$
To talk about \eqref{cptf.i} side, \cite{AE} among other related references use the {\it log KSBA} (Koll\'ar-ShepherdBarron-Alexeev) 
theory, which depends on the log minimal model program (MMP, for short). 
That is, for some general family of relative Cartier divisors $\mathcal{D}\subset \mathcal{U}_{2d}$ 
such that $(\mathcal{D}|_{X}) \in |m\mathcal{L}|_X|$ for any $(X,L)\in \mathcal{F}_{2d}$ with uniform $m\in \mathbb{Z}_{>0}$, 
we apply the theory of log KSBA (cf., \cite{She, AET} etc) to $(\mathcal{U}_{2d},c\mathcal{D})$ for some $(0<)\epsilon\ll 1$. 
Loosely speaking, it takes the so-called relative log canonical models of $(\overline{\mathcal{U}_{2d}},c\overline{\mathcal{D}})$ which is some extension 
of $(\mathcal{U}_{2d},c\mathcal{D})$, allowing certain base changes. This relative lc model satisfies certain canonicity. 
Hence, in the log KSBA theory, such arguments ensure the existence of 
a proper Deligne-Mumford stack  with projective coarse moduli $\overline{F_{2d}}^{\mathcal{D}}$ (\cite{Fjn, Kov}). 
The main new definition of \cite{AE} is the following. 
We closely follow the original expression and leave the actual necessary property 
in 
\begin{Def}[{\cite[Definition 6.2]{AE}}]
A family of divisors $\mathcal{D}$ is called {\it recognizable} if for any Kulikov model $X$ and its smoothing $\mathcal{X}$ which admits a primitive 
polarization $\mathcal{L}$ of degree $2d$, the limit divisor of $\mathcal{D}$ in $X$ is 
independent of $\mathcal{X}$ 
(in a certain weak sense
\footnote{in the sense of Claim \ref{kul}\eqref{rec} in the next section \S \ref{sec.AE}}). 
\end{Def}

\begin{Thm}[{\cite[Theorem 1.2]{AE}}]\label{AErecog}
If $\mathcal{D}$ is recognizable then, for the obtained compactified moduli of K3 surfaces $\overline{\mathcal{F}_{2d}}^{\mathcal{D}}$, 
its normalization $(\overline{\mathcal{F}_{2d}}^{\mathcal{D}})^{\nu}$ 
is the semitoric compactification for some semifan $\Sigma$ (\cite{Loo4}). 
\end{Thm}

We briefly outline this paper. 
In section \S \ref{sec.AE}, we sketch the proof of above Theorem \ref{AErecog} 
after slight simplification. In the next section \S \ref{sec.semitoric}, we give more general statement on the 
characterization of semitoroidal compactifications as log minimal models. 
Combining with our previous \cite{logmin}, 
this gives fairly different proof from the original proof of \cite[1.2]{AE}. The last section discusses a question on generalizing these 
semitoroidality phenomenon. In particular, we study polarized Enriques surfaces of degree $2$ after 
\cite{Shah.Enriques} and \cite{Sterk.I, Sterk.II}, showing compatibility with log K-stability viewpoint \cite{ADL0}.

\section{Reviewing sketch of Alexeev-Engel theorem \cite{AE}}\label{sec.AE}

In this section, we review the sketch of the original proof of \cite[1.2]{AE}($=$Theorem \ref{AErecog}) 
following my survey talk at the Kinosaki symposium in October 2021 (\cite{Kinosaki}). 
This section may hopefully serve as a (not-self-contained) survey on the breakthrough \cite{AE} and related papers. 

We focus on the proof of Theorem \ref{AErecog}($=$\cite[1.2]{AE}). 
Note that we make a slight simplification of the argument of {\it op.cit}, in Step \ref{step:AMRT.KSBA} in particular, 
while also occasionally mentioning how to obtain related (often more explicit) results. 
Our following explanation depends on version 2 of \cite{AE}. 

\begin{proof}[Sketch proof of Theorem \ref{AErecog}, with slight simplification]

We separate into several steps. 

\begin{Step}
The first crucial step is to construct a certain Kuranishi family of Kulikov degenerations, 
as indeed \cite{ABE, AET, AE, AE2, AEH} all depend on Kulikov degenerations. 
More precise statement in \cite{AE} is as follows: 
take any isotropic vector $e\in \Lambda_{2d}$ with $I:=\mathbb{Z}e$, 
and any primitive element $\lambda\in e^{\perp}(\subset \Lambda_{2d})/I$ with $\lambda^2\ge 0$, 
which we call a monodromy vector. 

Then note that \cite{AMRT} gives a partial toroidal compactification $\overline{(\mathcal{F}_{2d})}^{\lambda}$ 
and its coarse moduli $\overline{F_{2d}}^{\lambda}$ with the boundary divisor $\Delta_{\lambda}$. 
Take a small enough analytic open neighborhood $S_{\lambda}$ of $\Delta_{\lambda}$ in a chart of 
$\overline{\mathcal{F}_{2d}}^{\lambda}$. Then the families to be used are as follows: 

\begin{Claim}\label{kul}
\begin{enumerate}
\item \label{kul1}
(\cite[5.3 $+$ \S 5B, \S 5E, \S 5F]{AE})
There is a flat proper family $\pi\colon \mathcal{X}\to S_{\lambda}$ from a smooth $\mathcal{X}$, 
which restricts to a topologically locally trivial family of simple normal crossing (SNC, for short) K-trivial surfaces on $\Delta_{\lambda}$. 
\item \label{rec}
The recognizability condition of $\mathcal{D}$ ensures that 
the closure of the divisor family  $\mathcal{D}|_{(\mathcal{X}\setminus \pi^{-1}\Delta_{\lambda})}$ 
in $\mathcal{X}$ is a flat family of curves over $S_{\lambda}$. We simply denote it by $\mathcal{D}|_{\mathcal{X}}$ as (hopefully harmless) abuse of notation. 
\item \label{kul2}
(\cite[Cor. 6.11]{AE})
Furthermore, for a given recognizable divisor family $\mathcal{D}$, we can (re-)take $\mathcal{X}$ by base change and Atiyah-type flops, so that 
$\mathcal{D}|_{\mathcal{X}}$ does not contain any strata of 
the SNC degeneration $\pi^{-1}(s)$ for any $s\in \Delta_{\lambda}$. 
\end{enumerate}
\end{Claim}

These are polarized restrictions of what the authors of {\it op.cit} calls  $\lambda$-family and divisorial $\lambda$-family respectively. 
The keys to the proof of Claim \ref{kul} in \cite{AE} is as follows. 
For \ref{kul} \eqref{kul1}, first $\mathcal{X}|_{\Delta_{\lambda}}$ is constructed using the well-known explicit determination of 
components of Kulikov degenerations ({\it op.cit} \S 5B for type III case and  \S 5E for type II case, in un-polarized setup). Then, the next step  
uses gluing of  d-semistable smoothing deformations of 
Kulikov surfaces \cite{dss, Fri, FriS}. Then we restrict the obtained family to a polarized one in {\it op.cit} \S 5F. 
Further, to show \ref{kul} \eqref{kul2}, the authors use globalization of careful composition of Atiyah flops and base changes after \cite[2.9]{Laza2}. 
\end{Step}

%%%%%%%%%%%%%%%%%%%

%%%%%%%%%%%%%%%%%%%

\begin{Rem} 
In \cite{AET, ABE}, the above type family is constructed explicitly 
and the corresponding monodromy is explicitly calculated. This is the key 
to their explicit determination of the semi-fans describing the compactification. 
In {\it op.cit}, the monodromy calculation depends on the ``bare hands" geometric arguments in 
symplectic topology, after Symington \cite{Sym} and Engel-Friedman \cite{EF}.  
\end{Rem}

%%%%%%%%%%%%%%%%%%%

%%%%%%%%%%%%%%%%%%%

\begin{Step}
Another ingredient is to show that type III strata of the log KSBA compactification do {\it not}  
contain complete curves, or even quasi-affine. This steps are explained in 
each of \cite{AET, ABE, AE}. 

In the case of \cite{AET}, the irreducible components 
log KSBA limits are explicitly 
described depending on the earlier works of \cite{AT, AT2} on 
the log Calabi-Yau surfaces with involutions.

Similarly, in the case of \cite{ABE}, 
the irreducible components ar explicitly described though the 
components are sometimes slightly different from those appeared 
in \cite{AT, AT2}. 

For proving the general existence statement \cite[1.2]{AE}, although it is written in {\it loc.cit}, 
actually this step is avoidable as the following argument in 
the next step \ref{step:AMRT.KSBA} shows. 
\end{Step}

%%%%%%%%%%%%%%%%%%%

\begin{Step}\label{step:AMRT.KSBA}

Now we wish to identify the normalization of the log KSBA compactification with toroidal or 
semitoric compactification. It is known to expert that 
the normalization of the log KSBA compactification, as the singularities of the degenerate surfaces are semi-log-canonical (hence, DuBois), 
dominates the Satake-Baily-Borel compactification (cf., e.g., \cite[3.15]{AET}). 
Hence, the next step is to show the following claim: 
\begin{Claim}\label{K.A}
For the above log KSBA compactification 
$\overline{\mathcal{F}_{2d}}^{\mathcal{D}}$, there is a toroidal compactification $\overline{\mathcal{F}_{2d}}^{\Sigma}$ (\cite{AMRT}) 
for some fan $\Sigma$ such that the birational map 
$\overline{\mathcal{F}_{2d}}^{\Sigma}
\to \overline{\mathcal{F}_{2d}}^{\mathcal{D}}$ is a morphism. 
\end{Claim}
In original \cite[especially \S 8]{AE}, the above claim is proven by using 
Torelli type theorem (\cite{Fri.recent}) for log Calabi-Yau surfaces, stratification of 
log KSBA moduli, discussion of combinatorial type of the stable pairs etc. 
Here, we give a slight modification of their proof, without using them, as follows. 
\begin{proof}
For the birational map between an arbitrary toroidal compactification  $\overline{\mathcal{F}_{2d}}^{\Sigma}$
and the log KSBA compactification, 
take a resolution of indeterminancy 
$\overline{\mathcal{F}_{2d}}^{\Sigma} \xleftarrow{f} \overline{\mathcal{F}_{2d}}^{\rm mid} \xrightarrow{g} 
\overline{\mathcal{F}_{2d}}^{\mathcal{D}}$ with normal $ \overline{\mathcal{F}_{2d}}^{\rm mid}$. 
% as follows. 
%Although the arguments are not hard and probably standard, we include here for the sake of completeness. 

%Take a very ample line bundle $L$ on $\overline{\mathcal{F}_{2d}}^{R}$, and denote by $U$ 
%the maximum open subset of $\overline{\mathcal{F}_{2d}}^{\Sigma}$ which is contained in (i.e., 
%has open immersion to) $\overline{\mathcal{F}_{2d}}^{R}$. Further, take a line bundle $L'$ on 
%$\overline{\mathcal{F}_{2d}}^{\Sigma}$ extending $L|_{U}$ such that 
%$H^0(L)$ generates $H^0(I_Z L')$ for a closed subscheme $Z\subset \overline{\mathcal{F}_{2d}}^{\Sigma}$. 
%Replace $L'$ by its twist $L'^{\otimes m}(D)$ with $m\in \mathbb{Z}_{>0}$ and a divisor $D$, whereas replacing 
%$L$ by $L^{\otimes m}$, we can and do assume $Z$ has its codimension bigger than $1$ (in $\overline{\mathcal{F}_{2d}}^{\Sigma}$). 
%Then we its normarlized blow up along $Z$, which we denote by $ \overline{\mathcal{F}_{2d}}^{\rm mid}\xrightarrow{f}
%\overline{\mathcal{F}_{2d}}^{\Sigma}$, dominates $\overline{\mathcal{F}_{2d}}^{R}$ by the basepointfreeness of 
%$L'$. Hence we obtain the diagram $\overline{\mathcal{F}_{2d}}^{\Sigma} \xleftarrow{f} \overline{\mathcal{F}_{2d}}^{\rm mid} \xrightarrow{g} 
%\overline{\mathcal{F}_{2d}}^{R}$. 

Next, take a $f$-exceptional divisor $E$ and its center $H$ (as a closed subset) in 
$\overline{\mathcal{F}_{2d}}^{\Sigma}$. By further subdivision of $\Sigma$ if necessary, 
we can and do assume $H$ intersect with the open strata of some divisor. Indeed, it is possible by the following reason. 
First, we replace $\Sigma$ to a regular subdivision. Then, if $H$ lies in higher codimensional strata $O$ which is the intersection of 
$\mathbb{Q}$-Cartier divisors $D_i$, there is a corresponding ray $\R_{\ge 0}v$ in the corresponding open cone $C(F)$ of $U(F)\otimes \R$ (
notation following \cite[Chapter III]{AMRT}). Then a regular subdivision of $\Sigma$ including $\R_{\ge 0}v$ has the desired property. 

Now, we take a general point $p$ of $H$ lying in the (divisorial) open toric strata, 
which we still denote as $O$, and consider $f^{-1}(p)$ in $E$. We denote the monodromy $\lambda$ 
corresponding to $O$. 
Note $f^{-1}(p)$ has positive dimension and its image $g(f^{-1}(p))=:Z$ by $g$ has still positive dimension. Otherwise, it contradicts to the construction of 
$f, g$. Therefore, $Z$ parametrizes (at least set-theoritically) a nontrivial family of log KSBA pairs. 

On the other hand, consider the family of 
Claim \ref{kul} \eqref{kul2} 
%, constructed in \cite[\S 4, \S 5, \S 6]{AE}, 
and its fiber $(X_p,D_p:=\mathcal{X}|_{X_p})$ over $p$. 
Since the corresponding log KSBA pair should be simply ${\rm Proj}\oplus_{m\ge 0}\mathcal{O}_{X_p}(mD_p)$, 
it contradicts to that $\dim(Z)>0$. This finishes the proof of 
Claim \ref{K.A}. 
\end{proof} 

\vspace{2mm}
The remained last step is to show that a compactification of $F_{2d}$ 
which lies between a toroidal compactification and the Satake-Baily-Borel compactification 
(SBB compactification, for short)
is a semitoric compactification. 
It is originally proved in \cite[Theorem 7.18]{AE}. 
The main point of their original proof was 
that normal image of morphism from toric variety (resp., abelian variety) 
 is still toric (resp., abelian variety) and to apply it to 
 the fibers over each boundary component. \end{Step}
 
 Our following Theorem \ref{Loo.bir} gives a stronger claim, which also leads to a 
 different proof of \cite[Main theorem 1.2]{AE} (see Cor.\ \ref{Loo.bir2}, Figure \ref{AE.Od}). 
 
 Indeed, if a normal compactification $\overline{F_{2d}}$ of $F_{2d}$ lies between 
 a toroidal compactification and the SBB compactification, 
 if we consider the closure of modular branch divisor $\overline{B_i} \subset \overline{F_{2d}}$ 
 of degree $d_i$, it follows from \cite[3.4, \S4]{Mum77} that $(\overline{F_{2d}},\sum_i \frac{d_{i}-1}{d_i}\overline{B_i})$ 
 is log crepant to the SBB compactification with the strict transform of the same boundary, hence a log canonical minimal model. 
 Therefore, we can apply the following Theorem \ref{Loo.bir}. 
\end{proof}

%%%%%%%%%%%%%%%%%%%%%%%%%%%%%%%%%%%%%%%%%%%%%%%%%%%%%%%%%%%%%%%%%%%%%

\section{Toroidal and semitoroidal compactifications via MMP}\label{sec.semitoric}

The following theorem characterizes the toroidal compactification \cite{AMRT} and semitoric compactification \cite{Loo1, Loo2, Loo3, Loo4} 
in the language of the minimal model program (MMP) and log discrepancies. We expect this bridge gives various applications, going from 
one side to the other which we dare not to expand too much here. The easier direction of the result is known by \cite[3.4, \S4]{Mum77} (cf., also \cite[3.2-3.5]
{Ale1}) so our main point is to show the other harder direction and applications. 

\begin{Thm}[Semitoric and toroidal compactifications via MMP]\label{Loo.bir}
Following the  notation of \cite{AMRT}, \cite[Chapter 2]{OO}, 
let $D$ be a Hermitian symmetric space $G/K$, where $G$ is a simple Lie group and $K$ 
is a compact real form. For an algebraic group $\mathbb{G}$ over $\mathbb{Q}$ such that 
$G=\mathbb{G}(\R)$, let $\Gamma$ be any arithmetic discrete group and 
$X=\Gamma\backslash D$ be the associated locally Hermitian symmetric space. 
We also consider the branch prime divisors $B_i\subset X$ 
of $D\to \Gamma\backslash D$ with branch degree $d_i$. 

Consider a normal projective compactification $\overline{X}\supset X$. 
We denote the closure of $B_i$ inside it as $\overline{B_i}$ for each $i$, 
and the sum of boundary prime divisor (with all coefficients $1$ as) $\Delta$. 
\begin{enumerate} 
\item \label{Lb1}
Suppose $\Gamma$ is neat. Then 
$X\subset \overline{X}$ is toroidal (projective) compactification \cite{AMRT} 
with respect to a regular fan $\Sigma$ if and only if 
$(\overline{X},\Delta)$ is 
divisorially log terminal 
$\mathbb{Q}$-factorial (good)\footnote{goodness follows from the proof and assertion. Same for below \eqref{Lb2}, \eqref{Lb3}}
log minimal model. 
Here, we call a fan $\Sigma$ regular if and only if each cone $\sigma (\subset C(F))$ of $\Sigma$ 
is spanned by a part of an integral basis of $U(F)_{\mathbb{Z}}$ (again following the notation of \cite{AMRT}). 

\item \label{Lb2}
Here, $\Gamma$ is not necessarily neat. Then 
$X\subset \overline{X}$ is toroidal (projective) compactification \cite{AMRT} 
with respect to a regular fan $\Sigma$ only if 
$(\overline{X},\sum_i \frac{d_i -1}{d_i} \overline{B_i}+\Delta)$ is 
quasi-divisorially log terminal 
$\mathbb{Q}$-factorial (good) log minimal model. 

\item \label{Lb3}
Here again, $\Gamma$ is not necessarily neat but 
assume $X$ is unitary type or orthogonal type (so that 
semitoric theory works for sure.) 
 Then 
$X\subset \overline{X}$ is semitoric (projective) compactification \cite{Loo3, Loo4} 
with respect to a semi-fan $\Sigma$ if and only if 
$(\overline{X},\sum_i \frac{d_i -1}{d_i} \overline{B_i}+\Delta)$ is 
log canonical (good) log minimal model. 
\end{enumerate}
Also, for general non-projective proper analytic compactifications $\overline{X}\supset X$, 
if parts of all the above claims \eqref{Lb1}, \eqref{Lb2}, \eqref{Lb3} hold, as far as 
log minimality does not require projectivity. 
\end{Thm}

We apply the above characterization to weak K-moduli of K-trivial or Calabi-Yau (in the most generalized sense) varieties. 
In \cite{logmin}, we defined weak K-moduli to be 
a Deligne-Mumford stack but later in our \S \ref{Artin.sec}, we 
discuss general Artin stack setup as suggested by Y.Liu. 
Combining the previous work \cite{logmin} and the above \ref{Loo.bir}, we have our desired 
moduli-theoretic implication: 

\begin{Cor}[{a variant of \cite[Theorem 1.2]{AE}}]\label{Loo.bir2}
Consider a K-moduli (cf. \cite{logmin}) $M$ 
of (log-terminal) K-trivial polarized varieties, which has an isomorphism as a period map with an Hermitian locally symmetric space $M\simeq \Gamma \backslash D$ 
of orthogonal type or unitary type (\cite{Loo3, Loo4}). 
Take a weak K-moduli compactification of $M$ and 
denote its normalization as $M\subset \overline{M}$. 
Under the same notation as \eqref{Loo.bir} as above, suppose further 
\begin{enumerate}
\item $(\overline{M},\sum_i \frac{d_i -1}{d_i}\overline{B_i})$ is log canonical. 
\item \label{vmhs} On the open part of a finite cover of each of its lc centers, there is an admissible 
variation of mixed Hodge structures. 
\end{enumerate}
Then, there is a semi-fan $\Sigma$ such that $\overline{M}$ is isomorphic to the associated 
semitoric compactification of Looijenga $\overline{M}^{\Sigma}$. 
\end{Cor}

Our assumption that the uniformizing symmetric space is of orthogonal type or unitary type should be, I believe, redundant after all. 
We put the assumption to make sure the details of the 
theory of semitoric compactifications works (\cite{Loo3, Loo4}) 
as well as \cite[7.18]{AE}, \cite[(part of) 3.23]{AEH} works. 

Note that in particular 
Corollary \ref{Loo.bir2} provides another approach and extends the main feature of 
Alexeev-Engel theorem \cite[1.2]{AE}. Indeed, the existence of their (divisor) 
$\lambda$-families implies the 
condition \eqref{vmhs}. 

\begin{proof}[proof of Corollary \ref{Loo.bir2}, assuming \ref{Loo.bir}]
If we put a general enough family of $\Q$-divisors in the polarization to the total space of the family, 
with small enough coefficients, 
we can regard $\overline{M}$ as a log KSBA moduli so that its projectivity follows from 
\cite{Kov, Fjn}. Thus 
the log minimality of $(\overline{M},\sum_i \frac{d_i -1}{d_i}\overline{B_i})$ 
follows from the above two conditions 
due to  \cite[Theorem 3.2 and part of 3.6]{logmin}. 
Hence the assertion follows from Theorem \ref{Loo.bir} \eqref{Lb3}. 
\end{proof}

\begin{proof}[proof of Theorem \ref{Loo.bir}]

Firstly, we can assume $D$ is 
not the upper half space $\mathbb{H}$, as otherwise the assertion clearly holds. 
This implies that the boundary of the Satake-Baily-Borel compactification does not 
contain prime divisors. 
Further, for general $\Gamma$, we can reduce the proof to the case when $\Gamma$ is neat i.e., 
reducing \eqref{Lb2} to \eqref{Lb1} and reducing \eqref{Lb3} to that of case when $\Gamma$ is neat 
after replacing it by a finite subgroup (cf., e.g., \cite[Chapter III, \S 7]{AMRT}). 
Indeed, by such a replacement, the concerned log pairs just become finite log-crepant covers. 
Hence, we can and do assume that $\Gamma$ is neat. 

Recall that the Satake-Baily-Borel compactification $\overline{X}^{\rm SBB}$ 
is a log canonical model (as $D$ is not $\mathbb{H}$) again by \cite[3.4, \S 4]{Mum77}. 
We denote the natural birational map from $\overline{X}$ as 
$\varphi\colon \overline{X}\dashrightarrow \overline{X}^{\rm SBB}$. 
A priori, it could be not morphism due to lack of log canonical 
abundance theorem in general, but later in this proof we shall confirm this is actually a morphism. 

Take a boundary prime divisor of $X\subset \overline{X}$ as $E$. 
Then, it follows that the discrepancy 
$a(E; \overline{X}^{\rm SBB})$ 
is not more than $-1$, hence we have 
$a(E; \overline{X}^{\rm SBB})=-1$. 
Indeed, otherwise if we pass to the 
common resolution of $\overline{X}$ and $\overline{X}^{\rm SBB}$ we get a contradiction to the 
negativity lemma \cite[3.39]{KM}. 

Now we consider the toroidal compactification $\overline{X}^{\Sigma}$ 
 of $X$ for an arbitrarily chosen fan $\Sigma$. Then from $a(E; \overline{X}^{\rm SBB})=-1$ and 
 log-crepant-ness of $(\overline{X}^{\Sigma}, \Delta) 
 \to \overline{X}^{\rm SBB}$ 
 (\cite[\S 3]{MumHir}), 
 we have 
 $a(E; (\overline{X}^{\Sigma}, \Delta))=-1$. 
Since the boundary $\Delta$ of $\overline{X}^{\Sigma}$ is toroidal, 
each $E$ 
appears as a prime divisors of the toroidal compactificaftion 
for {\it some} fan and is corresponding to some ray $l_E\subset C(F)$ 
(\cite[3.5]{ops}). Further, we take the dual complex of the boundary of 
$\overline{X}$ which gives rise to a regular fan  
$\Sigma'$ whose set of rays modulo $\Gamma$ coincides with $\{l_E\}$. 
From the construction, we have a 
small birational morphism $\overline{X}\dashrightarrow 
\overline{X}^{\Sigma'}$ which preserves all the generic points of 
lc centers i.e., satisfying the second assumption of \cite[1.1]{Has19}. 
Therefore, we can apply {\it loc.cit} Theorem 1.1 to have a series of 
flops (log flips) $\overline{X}=X\xrightarrow{\varphi_0}\overline{X}_1
\xrightarrow{\varphi_1}\overline{X}_2 \cdots 
 \overline{X}_l=\overline{X}'$ followed by 
 an extremal contraction $\overline{X}' \to 
 \overline{X}^{\Sigma'}$. This 
connects $\overline{X}$ and $\overline{X}^{\Sigma'}$. Note that 
small extraction $f$ of {\it loc.cit} Theorem 1.1 is identity in our case because of the 
$\mathbb{Q}$-factoriality assumption on our $\overline{X}$ by \eqref{Lb3}. 
Similarly, $f'$ of {\it loc.cit} Theorem 1.1 is also identity since 
our $\Sigma'$ is regular so that $\overline{X}'$ is smooth (and hence 
$\mathbb{Q}$-factorial). 

We show
\footnote{There are two ways to complete logics from here. The other is to 
see the construction of $\overline{X}_i$ as log flips in \cite[\S 3]{Has19} and then 
confirm the open subset $U$ (in {\it loc.cit}) 
is preserved during the flops. Then we get immediate conclusion that 
$\overline{X}$ coincides with $\overline{X}^{\Sigma'}$. We thank 
K.Hashizume for the confirmation.}
 each $\overline{X}_i$ is toroidal compactification by induction on $l-i=j$ i.e., 
in a backword direction. 
The start point i.e., $\overline{X}_l$ is assumed to be 
a toroidal compactification $\overline{X}^{\Sigma'}$. 
Consider the flop $\overline{X}_i\dashrightarrow \overline{X}_{i+1}$ and suppose it is a 
log flip with the $\mathbb{Q}$-divisor $D'_i$ of $\overline{X}_i$. We denote the 
strict transform of $D'_i$ in $\overline{X}'_{i+1}$ as $D_i$. 
From \cite[Chapter III]{AMRT} and our assumption by induction that $\overline{X}_{i+1}$ is 
a toroidal compactification (for a neat discrete subgroup), it follows that each boundary prime divisor of 
it is a toric variety-fibration over an abelian variety fibration over another 
locally Hermitian symmetric space, say $(\Gamma\cap N(F))\backslash F$ 
after the notation of \cite[Chapter III]{AMRT}. 
The flopped contraction $\overline{X}_{i+1}\to V_i$ (cf., \cite[1.1]{Has19}) restricted to 
any boundary prime divisor is birational and 
must be fiberwise, i.e., preserving the existence of a morphism to 
$(\Gamma\cap N(F))\backslash F$, 
 and curves contracted by it is numerically 
determined i.e., as those orthogonal to $D_i$. 
Recall that every curve of a toric variety is algebraically equivalent to a torus-invariant $1$-cycle. 
Hence the flopped contraction is fiberwise toric and locally trivial, 
so that $\overline{X}_i$ is again a toroidal compactification. 
Therefore, by induction, it follows that $\overline{X}$ is a toroidal compactification. 
From the dlt condition (cf., \cite[3.9]{Fjn.lt}) 
together with neatness of $\Gamma$, it follows that it corresponds to 
a regular fan. 
\end{proof}

\begin{Ques}
Do the integral models of Shimura varieties as constructed in 
\cite{FC, Lan, Vasiu, Kisin} satisfy (relative) log minimality in the sense of 
Theorems \ref{Loo.bir}, \ref{Loo.bir2}? 
\end{Ques}

\begin{Rem}
Many of the above integral models $\mathcal{S}$ 
of Shimura varieties $S$, say over the ring of integers 
$\mathcal{O}$, 
satisfy a universality on extendability that 
any smooth scheme $\mathcal{T}$ over $\mathcal{O}$ and a morphism from 
the generic fiber 
$\mathcal{T}_{\rm Frac{\mathcal{O}}}$ 
to the Shimura variety $S$ extends to a morphism from 
$\mathcal{T}$ to $\mathcal{S}$. Note that this condition is of much stronger sort than scheme to be 
the complement of boundary of log minimal models whose  boundary coefficients are all $1$. 
For instance, take a projective equisingular family of minimal surfaces with $A_1$-singularities, 
which we simultaneously resolve. Any elementary transform of the union of exceptional divisors 
violates the extendability universality. 
\end{Rem}

%%%%%%%%%%%%%%%%%%%%%%%%%%%%%%%%%%%%%%%%%%%%%%%%%%%%%%%%%%%%%%%

\section{Allowing continuous isotropy -examples}\label{Artin.sec}

In Theorem \ref{Loo.bir2}, we discussed the structure of 
weak K-moduli which are {\it Deligne-Mumford} stacks, hence in particular 
only finite isotropy groups are allowed. 
A natural question is what happens if we allow continous isotropy to the 
weak K-moduli i.e., allow them to be Artin stacks. 

\begin{Conj}\label{weak.Artin}
For any weak K-moduli {\it Artin stack} of polarized irreducible holomorphic symplectic varieties or polarized abelian varieties or 
polarized Enriques manifolds (\cite{OS}), 
the normalization of the coarse moduli space (if it exists) is always a semitoric compactification. 
\end{Conj}

We provide some non-trivial examples as an evidence, 
which led us to the above question. For more details of the 
surfaces to be parametrized, we refer to the original papers for each.

\begin{Rem}

Also, we note that they are also qualitively different in the sense 
that the Shah-Looijenga's is a good moduli {\it Artin stack} 
in the sense of J.Alper 
(see also a sibling K-moduli of Del Pezzo surfaces of degree $2$ in \cite{OSS})
while the \cite{AET} is a Deligne-Mumford stack. 
Indeed, recall that \cite{Shah}'s compactification is a one 
point blow up of the natural GIT compactification of the moduli 
of sextic curves, but 
two of the four $1$-cusps are with $1$-dimensional stabilizer groups. 
Note that the partial Kirwan type blow up in \cite{Shah} 
can be done in stacky level as 
following \cite[\S 5.2]{OSS}. 

Recently, there is also a work of \cite{AET}. 
Here, we rigorously confirm that the compactification of $\mathcal{F}_{2}$ by \cite{AET} is different from the compactificaftion by 
\cite{Shah} (cf., also 
\cite{Loo4}). 
It is easy if we compare the morphisms to Baily-Borel compactification: 
the morphism from the former compactification contracts various divisors 
(corresponding to the elliptic diagrams of type e.g., 
$'A_{18}^{-}$, $D_{18}$, $'A'_{15}A_{3}$ etc in 
\cite[Table 4.2]{AET}) to the 0-cusp while 
it is small in the case of Shah-Looijenga compactification. 
\end{Rem}

\begin{Ex}
\cite{ABE} explicitly constructs a Deligne-Mumford moduli stack of 
(Weierstrass) elliptic K3 surfaces and their slc K-trivial degenerations. 
They showed the coarse moduli is a particular toroidal compactification, 
as predicted by \cite{Brun}. It is reproved (reconstructed) by 
a different method, just by Weierstrass equations without use of 
Kulikov models, in \cite[Part I]{PL} which naturally extends over 
$\mathbb{Z}[1/6]$. The method has a virtue to apply to 
a differential geometric problem of deciding all limit measures 
along type II degenerations \cite[Part II]{PL}. 

On the other hand, as \cite[Chapter 7]{OO} shows, 
the Satake-Baily-Borel compactification coincides with GIT compactification 
whose boundary still parametrizes slc K-trivial degenerations, hence weak K-moduli. 
Note that it is an Artin stack at the level of stack, hence giving another affirmative 
example to Problem \ref{weak.Artin}. 
\end{Ex}

\begin{Ex}[Degree $2$ K3 surfaces]\label{F2.ex}
Recall that \cite{Shah, Loo1, Loo4} showed that a certain 
weak K-moduli Artin stack has a coarse moduli which is 
a semitoric compactification corresponding to a simple hyperplane arrangement. 
Note that generic degree $2$ K3 surface is obtained as the double cover of 
projective plane branched along a sextic curve. Then, the 
desired weak K-moduli is obtained as blow up of the GIT moduli, 
whose exceptional divisor parametrizes double branched covers of 
$\mathbb{P}(1,1,4)$. 

For the precise construction of the moduli {\it stack}, which is missing in the original papers, 
see \cite[\S 5.2]{OSS}. {\it Loc.cit} 
explicitly determines the K-moduli {\it stack} of Del Pezzo surfaces of degree $2$. 
Note that they are similarly  double covers of 
projective plane branched along a quartic curve, 
while an exceptional divisor parametrizes doubled $\mathbb{P}(1,1,4)$ again. 
Hence, the completely same stacky refinement of the above compactification works. 
\end{Ex}

\begin{Ex}[Hyperelliptic quartic K3 surfaces]\label{F4.ex}

Recall that general hyperelliptic quartic K3 surface is a double cover of 
quartic surface $\mathbb{P}^1\times \mathbb{P}^1$ ramifying 
along a $(4,4)$-curve $D$. 
As a related recent work, 
\cite{ADL0} explicitly described log K-moduli of log Fano pairs 
$(\mathbb{P}^1\times \mathbb{P}^1, c D)$ with $c\in (0,1)$ 
and their degenerations, 
giving the K-stability perspectives to the explicit prediction of 
birational variation of compactifications by 
\cite{LO} and its partial confirmation \cite{LO.confirm} 
(also cf., \cite{Shah2}). All the appearing models in {\it op.cit} 
can be naturally considered as birational models of $F_4$. 

As a precise result, \cite[Theorem 1.2]{ADL0} together with 
\cite[Theorem 1.1 (iv)]{LO.confirm} show that 
the log K-moduli for $c=1-\epsilon$ with $|\epsilon|\ll 1$ coincides with 
the semitoric compactification corresponding to the closure of a Heegner divisor 
parametrizing degree $8$ curves of $\mathbb{P}(1,1,2)$ 
generally not passing through the vertex (cf., \cite{Loo4} for the precise meaning). 
We denote this K-moduli stack by $\overline{\mathcal{F}_4}^{\Sigma}$. 
This is, at the level of Artin stack, another example of weak K-moduli stack of 
polarized K3 surfaces whose coarse moduli is semitoric. 
\end{Ex}

We treat the following case of the moduli of  degree $2$ Enriques surface after 
\cite{Hor, Shah.Enriques, 
Sterk.I, Sterk.II} with some details because we more substantially add new viewpoints to {\it loc.cit} 
(especially Proposition \ref{ADL.SS}). 

\begin{Ex}[Enriques surfaces of degree $2$]

After Horikawa \cite{Hor}, degree $2$ polarization on Enriques surface 
gives a description as the double cover branched along $2$-anticanonical 
curves in quadric surfaces. We denote the moduli stack by  
$\mathcal{M}_{\rm Enr}$ and its coarse moduli space $M_{\rm Enr}$. 
By using the Horikawa(-Enriques)'s description, 
\cite{Shah.Enriques} proves some stable reduction theorem type result 
and \cite{Sterk.I, Sterk.II} clarified it as the valuative criterion of properness for 
certain compactification of the moduli. I am grateful to Professor S.Mukai as I have learnt a lot about this compactification from him. 
We now review \cite[\S 1]{Sterk.II} and \cite{Shah.Enriques} somewhat: 
he considers the series of quotient stacks 
\begin{align}\label{stack.morphs}
[H^{\rm ss}_2/G]\to [H^{\rm ss}_1/G]\to [H^{\rm ss}/G].
\end{align}
Here, $H$ is the closed subspace of $|\mathcal{O}_{\mathbb{P}^1\times \mathbb{P}^1}
(4,4)|$ fixed by the natural involution $I$, on 
which the algebraic group $G:={\rm Aut}(\mathbb{P}^1\times \mathbb{P}^1,I)$ acts. 
The identity component of $G$ is $2$-dimensional algebraic torus. 
The above morphisms \eqref{stack.morphs} 
induce blow ups at the level of coarse moduli spaces as follows: 
$$H^{\rm ss}_2//G \xrightarrow{\pi_2} H^{\rm ss}_1//G \xrightarrow{\pi_1} H^{\rm ss}//G.$$
We denote the exceptional divisor of $\pi_i$ as $E_i$. 
The main result of \cite{Sterk.II} are that 
\begin{itemize}
\item $H^{\rm ss}_2$ underlies a $G$-equivariant proper flat family of 
polarized Enriques surfaces of degree $2$ and their slc K-trivial 
degenerations (after \cite[\S6, \S7]{Shah.Enriques}) 
\item The projective and normal GIT quotient 
$H^{\rm ss}_2//G$ birationally 
dominates a semitoric compactification 
$\overline{M}_{\rm Enr}^{\Sigma}$, associated to the strict transform of $E_1$ 
generically parametrizing ``special" Enriques surfaces in the sense of 
Horikawa \cite{Hor} (which are also nodal Enriques surfaces in our polarized setup). 
\end{itemize}

Note that, by using multiple of the polarizations on the parametrized surfaces on $H_2$, 
there is a proper morphism to a corresponding Hilbert scheme $H_2\to {\rm Hilb}$. 
This morphism contracts the second exceptional divisor $E_2$ (i.e., that of 
$H_2\to H_1$) and we write the normalization of the image of $H_2^{\rm ss}$ as $H_3^o$, 
on which $G$ still acts. 

Then we consider the moduli {\it Artin} stack $[H_3^o/G]$ and denote 
it by $\overline{\mathcal{M}}_{\rm Enr}^{\Sigma}$ with its coarse moduli 
$\overline{M}_{\rm Enr}^{\Sigma}$. 
The notation comes from the following, which provides modern different explanation of the works of 
\cite{Shah.Enriques, Sterk.II}. 
\begin{Prop}\label{ADL.SS}
$\overline{M}_{\rm Enr}^{\Sigma}$ is the closure of $M_{\rm Enr}$ in $\overline{F_4}^{\Sigma}$ 
(cf., Example \ref{F4.ex}). 
Moreover, there is a natural closed immersion of $\overline{\mathcal{M}}_{\rm Enr}^{\Sigma}$ into 
$\overline{\mathcal{F}_4}^{\Sigma}$ which also respects the 
parametrized polarized surfaces. 
\end{Prop}

\begin{proof}
We consider the closure of $\mathcal{M}_{\rm Enr}$ in 
the log K-moduli stack $\overline{\mathcal{F}_4}^{\Sigma}$, 
and denote it by $\overline{\mathcal{M}_{\rm Enr}}^{\rm ADL}$. 
Then, we claim the following: 
\begin{Claim}\label{op.im}
$\overline{\mathcal{M}_{\rm Enr}}^{\rm ADL}$ has an open 
immersion to $[H_3^o/G]$, respecting the parametrized surfaces. 
\end{Claim}
\begin{proof}[proof of Claim\ref{op.im}]
This can be shown as follows. 
Note that the former algebraic stack parametrizes log K-semistable (resp., -polystable, 
at the coarse moduli level) and the latter parametrizes GIT-semistable 
(resp., -polystable, at the coarse moduli level) objects, in each strata. 
By \cite[Theorem 1 and its proof]{LP}, 
$(\mathbb{P}(1,1,2), (1-\epsilon)D)$ for general $D\in |-K_X/2|$ 
is log K-semistable. Also, by \cite[6.1 (ii)]{OS}, 
it can be log K-unstable for special such $D\in |-K_X/2|$ such as a line with 
multiplicity $4$. 
After these two confirmations, we can and do apply \cite[3.4 (also cf., 3.6)]{OSS} to conclude the presence of the open 
immersion $\tilde{\iota}$ 
of 
$\overline{\mathcal{M}_{\rm Enr}}^{\rm ADL}$ (which parametrizes log K-semistable pairs) 
into $\overline{\mathcal{M}}_{\rm Enr}^{\Sigma}$ (which parametrizes GIT semistable curves in each strata), 
as well as its induced morphism $\iota$ between the coarse moduli spaces 
$\overline{M_{\rm Enr}}^{\rm ADL}$ (which parametrizes log K-polystable pairs) 
into $\overline{M}_{\rm Enr}^{\Sigma}$ (which parametrizes GIT polystable curves in each strata). 
We conclude the proof of Claim\ref{op.im}. 
\end{proof}

From the construction, both morphisms $\tilde{\iota}$ and $\iota$ respect the parametrized polarized surfaces. 
Moreover, since the above morphism between the two coarse moduli are isomorphism, 
and (log) K-semistability is a Zariski open condition \cite{Od.open, Don.open, BLX}, 
the morphism between algebraic {\it stacks} $\overline{\mathcal{M}_{\rm Enr}}^{\rm ADL}\to [H_3^o/G]$ is also isomorphism. 
Another proof goes as follows which we omit the details. 
We first directly analyze the (simple) contraction of $E_2\subset H_2^{ss}//G$, 
different from $\pi_2$ which underlies the morphism of algebraic stacks 
$[H_2^{ss}/G]\to [H_3^o/G]$. Denote the contraction as 
$\pi'\colon H_2^{ss}//G\to \overline{M_{\rm Enr}}'$. Then identify the obtained 
$\overline{M_{\rm Enr}}'$ as the semitoric compactification discussed in 
\cite{Sterk.II, Loo4}, by explicitly confirming the relative ampleness of the strict transform of $E_1$ 
over the Satake-Baily-Borel compactification. 
\end{proof}

\begin{Rem}
The above proposition gives a better understanding to \cite[Remark 4.5]{Sterk.II}. 
Indeed, we see from the above theorem that the contraction of the Shah-Sterk's compactification to semitoric compactification as claimed in 
\cite[Main Theorem, also cf., Remark 4.5]{Sterk.II} is nothing but the induced map by 
the contraction $H_2^{ss}\to H_3^o$ and hence has natural modular interpretation. 
That is, the semitoric compactification discussed in \cite{Sterk.II} parametrizes 
polarized Enriques surfaces and numerically K-trivial 
degenerations with slc (``insignificant limit") singularities. This is the reason why we discuss this example in our context. 
Explicitly speaking, this contraction contracts $E_2$ to the $8$-dimensional GIT moduli discussed in \cite[Theorem 6.3]{Shah.Enriques} 
inside the strict transform of $E_1$, 
which finally replaces the point in $E_1$ corresponding to \cite[p.489, l.23-25]{Shah.Enriques}. 
This also geometrically overcomes the problem: 
\begin{quote}
``{\it the price we have to pay (for the moment) is that all fibers ... describe the same GIT space of curves..}" (\cite[p130, l.8-9]{Sterk.thesis})
\end{quote}
\end{Rem}

\begin{Rem}
We also take this chance to write that \cite[B-6$b_1$]{Shah.Enriques} is missing 
the possibility of $L=L'$, corresponding to the 
union of two skew horizontal edges and two vertical lines which are {\it not} edges. 
Also, {\it op.cit} B-2 of p485 seems to mistake the logic. The correct 
reasoning comes from the presence of isotrivial degeneration of the curve into the 
missed curve of {\it op.cit} B-6 $b_1$. The former missed case fits to the fourth 
$0$-cusp in the setup of \cite{Sterk.I, Sterk.II} which is not depicted in 
$3$ Figures of \cite[\S 4]{Sterk.II}. 
\end{Rem}

\end{Ex}

%%%%%%%%%%%%%%%%%%%%%%%%%%%%%%%%%%%%%%%%%%%%%%%%%%%%%%%%%%%%%%%%%%%%%%%%%%%%%%%%%%%%%%%%%%%%%%%%%%%%%%%%%%%%%%%%

\subsection*{Acknowledgements}
I thank V.Alexeev, P.Engel, 
K.Hashizume, S.Kondo, Y.Liu, 
E.Looijenga, and S.Mukai for helpful related discussions. 
This paper is partially an outcome of my preparation for the survey talk at 
Kinosaki symposium 2021. I therefore also thank the organizers, 
and those who suggested english translation of the proceeding. 

During this work, the author is partially supported by 
KAKENHI 18K13389 (Grant-in-Aid for Early-Career Scientists), 
KAKENHI 16H06335 (Grant-in-Aid for Scientific Research (S)), and 
KAKENHI 20H00112 (Grant-in-Aid for Scientific Research (A)). 
%%%%%%%%%%%%%%%%%%%%%%%%%%%%%%%%%%%%%%%%%%%%%%%%%%%%%%%%%%%%%%%%%%%%%
%%%%%%%%%%%%%%%%%%%%%%%%%%%%%%%%%%%%%%%%%%%%%%%%%%%%%%%%%%%%%%%

%%%%%%%%%%%%%%%%%%%%%%%%%%%%%%%%%%%%%%%%%%%%%%%%%%%%%%%%%%%%%%%%%%%%%%%%
%%%%%%%%%%%%%%%%%%%%%%%%%%%%%%%%%%%%%%%%%%%%%%%%%%%%%%%%%%%%%%%%%%%%%%%%

%%%%%%%%%%%%%%%%%%%%%%%%%%%%%%%%%%%%%%%%%%%%%%%%%%%%%%%%%%%%%%%

\vspace{5mm}
\footnotesize \noindent
{\tt yodaka@math.kyoto-u.ac.jp} \\
Department of Mathematics, Kyoto university, Kyoto, Japan \\


\begin{thebibliography}{FGA}


%\bibitem[Ale02]{Ale02}
%V.~Alexeev, 
%Complete moduli in the presence of semiabelian group action, 
%Ann. of Math 155 (2002), 611-708. 

\bibitem[Ale96]{Ale1}
V.~Alexeev, 
Log canonical singularities and complete moduli of stable pairs, arXiv 
alg-geom/9608013.

%\bibitem[AN99]{AN}
%V.~Alexeev, I.~Nakamura,
%On Mumford's construction of degenerating abelian varieties, 
%Tohoku Math.\ J.\ vol.\ 51, pp.399--420 (1999). 

%\bibitem[Ale06]{Ale2}
%V.~Alexeev, 
%Higher-dimensional analogues of stable curves,
%Proceedings of Madrid ICM2006, 2 (2006), 515-536, Eur. Math. Soc. Pub. House.




%%

\bibitem[AET19]{AET}
V.~Alexeev, P.~Engel, A.~Thompson, 
Stable pair compactification of moduli of K3 surfaces of degree 2, 
arXiv:1903.09742.

\bibitem[ABE20]{ABE} 
V.~Alexeev, A.~Brunyate, P.~Engel, 
Compactifications of moduli of elliptic K3 surfaces: stable pairs and 
toroidal, arXiv:2002.07127v5. to appear in Geometry and Topology. 

%%


\bibitem[AE21a]{AE}
V.~Alexeev, P.~Engel, 
Compact moduli of K3 surfaces, 
arXiv:2101.12186v2. 

\bibitem[AE21b]{AE2}
V.~Alexeev, P.~Engel, 
The flex divisor of a K3 surface, 
arXiv:2109.14603v1. 


\bibitem[AEH21]{AEH}
V.~Alexeev, P.~Engel, C.Han, 
Compact moduli of K3 surfaces with a nonsymplectic 
automorphism, arXiv:2110/13834v1. 


\bibitem[AT15]{AT}
V.~Alexeev, A.~Thompson, 
Modular compactification of moduli of K3 surfaces of degree 2, 
preprint available at the first author's website. 

\bibitem[AT17]{AT2}
V.~Alexeev, A.~Thompson, 
ADE surfaces and their moduli, 
arXiv:1712.07932. 

\bibitem[ADL20]{ADL0}
K.Ascher, K.DeVleming, Y.Liu, 
K-moduli of curves on a quadric surface and K3 surfaces, 
arXiv:2006.06816v1. 

%\bibitem[ADL21]{ADL}
%K.Ascher, K.DeVleming, Y.Liu, 
%K-stability and birational models of moduli of quartic K3 surfaces
%arXiv:2108.06848. 


\bibitem[AMRT]{AMRT}
A. Ash, D. Mumford, M. Rapoport, Y.-S. Tai, Smooth compactifications 
of locally symmetric varieties, Cambridge Mathematical Library, 
Second edition (2010). 

%\bibitem[BorJi]{BJi}
%A.~Borel, L.~Ji, 
%Compactifications of symmetric and locally symmetric spaces, 
%Mathematics: Theory \& Applications. Birkh\"auser, (2006). 

\bibitem[BLX19]{BLX}
H.Blum, Y.Liu, C.Xu, 
Openness of K-semistability for Fano varieties, 
Duke Mathematical Journal 168 (11), 2029-2073, 2019. 

\bibitem[Brun15]{Brun}
A.~Brunyate, 
A modular compactification of the space of elliptic K3 surfaces, 
UGA Ph.D thesis (2015). 

%\bibitem[Carl80]{Car}
%J.~Carlson, 
%%Extension of mixed Hodge structures, 
%In: Journ\'ees de G\'eometrie Alg\'ebrique d'Angers, 
%Juillet 1979/Algebraic geometry, Angers, 1979. 
%1980, 107-127. 

%{\red{ThmB: K3}}

\bibitem[Don15]{Don.open}
S.Donaldson, 
Algebraic families of constant scalar curvature K\"ahler metrics, 
arXiv:1503.05174 

%\bibitem[Dol08]{Dolg.ref}
%I.Dolgachev, 
%Reflection groups in algebraic geometry, Bull. Amer. Math. Soc. 45 (2008). 

%\bibitem[Eng18]{Eng}
%P.Engel, 
%Looijenga's conjecture via integral-affine geometry, 
%J.\ Diff.\ Geom.\ 109 (2018), no.3, 467-495. 

\bibitem[EF19]{EF}
P.Engel, R.Friedman, 
Smoothings and rational double point adjacencies for 
cusp singularities, J. Differential Geom. (2019). 


%\bibitem[Ess96]{Ess}
%F. Esselmann, ``Uber die maximale Dimension von Lorentz-Gittern mit co- endlicher %Spiegelungsgruppe", J. Number Theory, 61 (1996), 103-144. 

\bibitem[FC90]{FC}
G.Faltings, C.-L. Chai, Degenerations of abelian varieties, vol. 22, Ergebnisse der Mathematik und ihrer Grenzgebiete, no. 3, Springer-Verlag, 1990.

\bibitem[Fjn18]{Fjn}
O.Fujino, 
Semipositivity theorems for moduli problems, Ann.\ of Math., 
pp. 639-665 from Volume 187 (2018) 

\bibitem[Fuj07]{Fjn.lt}
O.Fujino, What is log terminal?, p.49-62, 
in the book ``Flips for 3-folds and 4-folds", 
Oxford University Press (2007)

\bibitem[Fri83]{dss}
R.~Friedman, 
Global smoothings of varieties with normal crossings, 
Ann.\ Math vol. 118 (1983). 


%\bibitem[FriMor83]{FM}
%R.~Friedman, D.~Morrison, 
%The birational geometry of degenerations,
%Progress in Math.\ 29, Birkh\"auser, (1983). 



\bibitem[Fri84]{Fri}
R.~Friedman, 
A new proof of the global Torelli theorem for K3 surfaces, 
Ann. of Math. (2) 120 (1984), no. 2, 237--269. 


\bibitem[FriSca86]{FriS}
R.~Friedman, F.~Scattone, 
Type III degenerations of $K3$ surfaces, Invent.\ Math.\ 83 (1986), no.~1, 1--39. 

\bibitem[Fri15]{Fri.recent}
R.~Friedman, 
On the geometry of anticanonical pairs, 
arXiv:1502.02560. 


\bibitem[FriSca86]{FriS}
R.~Friedman, F.~Scattone, 
Type III degenerations of $K3$ surfaces, Invent.\ Math.\ 83 (1986), no.~1, 1--39. 


\bibitem[GHKS]{GHKS}
M.~Gross, P.~Hacking, S.Keel, B.Siebert, 
Theta functions and K3 surfaces, in preparation. 

\bibitem[Has19]{Has19}
K.Hashizume, 
Relations between two log minimal models of log canonical pairs, 
Int.\ J.\ Math.\ 31 (13), 2050103. 

\bibitem[Hor78]{Hor}
E.Horikawa, 
On the periods of Enriques' surfaces, II., 
Math.\ Ann.\ 235., pp.217-246 (1978). 


%\bibitem[HLL20]{HLL}
%Klaus Hulek, Christian Lehn, Carsten Liese, 
%On the GHKS compactification of the moduli space of K3 surfaces of degree two, arXiv:2010.06922 

\bibitem[Kis10]{Kisin}
M. Kisin, Integral models for Shimura varieties of abelian type, J. Amer. Math. Soc. 23 (2010), no. 4, 967-1012. 

%\bibitem[Kob90b]{Kobb90}
%R.~Kobayashi, 
%Ricci-flat K\"ahler metrics on affine algebraic manifolds 
%and degenerations of K\"ahler-Einstein K3 surfaces, 
%In \textit{K\"ahler Metrics and Moduli Spaces}, 257--311, 
%Adv.\ Stud.\ Pure Math., 18-II, T.~Ochiai. ed.\ Academic Press, (1990). 

%\bibitem[Kol]{Kol}
%J.\ Koll\'ar, 
%A book on KSBA moduli in preparation. 
%a draft available at {\texttt{https://web.math.princeton.edu/\~{}kollar/}} 

\bibitem[KolMor98]{KM}
J.~Koll\'ar, S.~Mori, 
Birational Geometry of Algebraic Varieties, 
Cambridge Tracts in Mathematics, vol.\ 134, 
Cambridge University Press, (1998). 


%\bibitem[Kul77]{Kul}
%V. Kulikov, Degenerations of $K3$ surfaces and Enriques surfaces, 
%Izv.\ Akad.\ Nauk SSSR Ser.\ Mat.\ 41 (1977), no.~5, 1008--1042, 1199.


\bibitem[KP17]{Kov}
S. Kov\'acs, Z. Patakfalvi, 
Projectivity of the moduli space of stable log-varieties 
and subadditivity of log-Kodaira dimension, J. Amer. Math. Soc. 30 (2017), no. 4, 
959-1021.

\bibitem[Lan13]{Lan}
K.-W. Lan, Arithmetic compactifications of PEL-type Shimura varieties, London Mathematical Society Monographs Series, vol. 36, Princeton University Press, Princeton, NJ, 2013.

\bibitem[Laza16]{Laza2}
R.~Laza, 
The KSBA compactification for the moduli space of degree two K3 pairs, J. Eur. Math. Soc. 18 (2016), no. 2, 225--279. 




\bibitem[LO19]{LO}
R. Laza, K.O'Grady, 
Birational geometry of the moduli space of quartic K3 surfaces, Compositio Math. 155 (2019), no. 9, 1655-1710.
%{\red{一般予想 １本目}}

\bibitem[LO21]{LO.confirm}
R. Laza, K.O'Grady, 
GIT versus Baily-Borel compactification for K3's which are double covers of 
$\mathbb{P}^1\times \mathbb{P}^1$, Adv. Math. 383 (2021). 

\bibitem[LP20]{LP}
Y.Liu, A.Petracci, 
On K-stability of some Del Pezzo surfaces of Fano index $2$, 
arXiv:2011.05094v3. 

\bibitem[Loo84]{Loo1}
E. Looijenga, New compactifications of locally symmetric varieties, in ``Proc. of the 1984
Vancouver Conference in Algebraic Geometry" (J. Carrell et al., eds.), CMS Conf. Proc, 6, A . M . S., Providence, 1986, pp.341-364.

\bibitem[Loo85]{Loo2}
E. Looijenga, 
Semi-toric partial compactifications I, Report 8520 (1985), 72 pp., Catholic
University Nijmegen. 

\bibitem[Loo03a]{Loo3}
E. Looijenga, 
Compactifications defined by arrangements. I, 
- The ball quotient case. Duke Math. J. 118 (2003), no. 1, 151-187. 


\bibitem[Loo03b]{Loo4}
E. Looijenga, 
Compactifications defined by arrangements. II, 
Locally symmetric varieties of Type IV, Duke Math. J. 119 (2003) 
no.3, 527-588. 




%\bibitem[Mum72b]{Mum72.AV}
%D.~Mumford, 
%An analytic construction of degenerating abelian varieties over complete local rings, 
%Compositio Math (1972). 

\bibitem[Mum77]{MumHir}
D.Mumford, 
Hirzebruch's proportionality theorem in the noncompact case, Invent. Math. 42 (1977) 



%\bibitem[MS84]{MS84}
%J.~Morgan, P.B.~Shalen, 
%Valuations, trees, and degenerations of hyperbolic structures.\ I, 
%Ann.\ of Math.\ (2) 120 (1984), no.~3, 401--476. 


%\bibitem[Mum72b]{Mum72.AV}
%D.~Mumford, 
%An analytic construction of degenerating abelian varieties over complete local rings, 
%Compositio Math (1972). 




\bibitem[Mum77]{Mum77}
D.~Mumford, 
Hirzebruch's proportionality theorem in the noncompact case,
Invent.\ Math.\ 42 (1977), 239--272. 



%\bibitem[Nak99]{Nak1}
%I.~Nakamura, 
%Stability of degenerate abelian varieties, 
%Invent.\ Math.\ vol.\ 136, pp.659--715 (1999). 

%\bibitem[Nak10]{Nak2}
%I.~Nakamura, 
%Another canonical compactification of the moduli space of abelian varieties, 
%Algebraic and arithmetic structures of moduli spaces (Sapporo, 2007), 
%Adv.\ Studies Pure Math.\ vol.\ \textbf{58}, pp.69-135 (2010). 

\bibitem[Odk10]{Od10}
Y.~Odaka, 
On the GIT stability of polarized varieties - a survey, 
Proceedings of the 2010 Kinosaki Symposium. 

\bibitem[Odk12a]{Od2}
Y.~Odaka, 
The Calabi conjecture and K-stability, 
Int.\ Math.\ Res.\ Not.\ IMRN (2012), no.~10, 2272--2288. 


\bibitem[Odk13a]{Od}
Y.~Odaka, 
The GIT stability of polarized varieties via Discrepancy, 
Ann.\ of Math.\ (2) 177 (2013), no.~2, 645--661. 

\bibitem[Odk13b]{Od.open}
Y.~Odaka, 
On the moduli of Kahler-Einstein Fano manifolds, 
Proceeding of Kinosaki algebraic geometry symposium 2013. arXiv:1211.4833

\bibitem[OS15]{OS}
Y.Odaka, S.Sun, 
Testing log K-stability via blow up formalism, 
Annales de la Facult\'{e} des Sciences de Toulouse. 
S\'erie 6,  Tome 24 (2015) no. 3,  p. 505-522. 

\bibitem[OSS16]{OSS}
Y.~Odaka, C.~Spotti, S.~Sun, 
Compact moduli spaces of Del Pezzo surfaces and K\"ahler-Einstein metrics, 
J. Differential Geom. 102 (1) (2016). 

%\bibitem[OO18a]{OO.announce}
%Y.~Odaka, Y.~Oshima, 
%Collapsing K3 surfaces and moduli compactification, 
%Proc.\ Japan Acad.\ Ser.\ A Math.\ Sci. 94 (2018), no.~8, 81--86.


\bibitem[OO21]{OO}
Y.~Odaka, Y.~Oshima, 
Collapsing K3 surfaces, Tropical geometry and Moduli compactifications of Satake and Morgan-Shalen type, MSJ. Memoir vol. 40 (2021)



\bibitem[Od20a]{ops}
Y.Odaka, 
Polystable log 
Calabi-Yau varieties and Gravitational instantons, 
to appear in Journal of Mathematical Sciences. 

\bibitem[Od20b]{PL}
Y.Odaka, 
PL density invariant for type II degenerating K3 surfaces, Moduli compactification and hyperKahler 
metrics,  to appear in Nagoya Journal of Math. 

\bibitem[Od21a]{logmin}
Y.Odaka, 
On log minimality of weak K-moduli compactifications of Calabi-Yau varieties, 
arXiv:2108.03832

\bibitem[Od21b]{Kinosaki}
Y.Odaka, 
On the recent development around algebro-geometric compactifications 
of the moduli of K3 surfaces - mainly after the work of Alexeev-Engel - (in Japanese), 
Kinosaki symposium 2021. p.166-213. 

\bibitem[OS11]{OS}
K. Oguiso, S. Schr\"oer: Enriques Manifolds. J. Reine Angew. Math. 661. (2011), 215-235 

\bibitem[PSS71]{PS.Sha}
I.I. Piatetski-Shapiro, I.R.~Shafarevich,
Torelli's theorem for algebraic surfaces of type K3, 
Izv.\ Akad.\ Nauk SSSR Ser.\ Mat.\ 35 (1971), 530--572, in Russian. 




%\bibitem[Sca87]{Sca}
%F.~Scattone, 
%On the compactification of moduli spaces for algebraic $K3$ surfaces, 
%Mem.\ Amer.\ Math.\ Soc.\ 70 (1987), no.~374.

%\bibitem[Sch73]{Schm}
%W.~Schmid, 
%Variation of Hodge structure: the singularities of the period mapping, 
%Invent.\ Math.\ 22 (1973), 211-319. 

%\bibitem[Ser06]{Sernesi}
%E.~Sernesi, 
%``Deformations of Algebraic Schemes", 
%Grundlehren der Mathematischen Wissenschaften, 334, 
%Springer-Verlag (2006). 

%\bibitem[Ses72]{Sesh}
%%C.S.~Seshadri, 
%Quotient spaces modulo reductive algebraic groups, 
%Ann.\ of Math.\ (2) 95 (1972), 511--556. 


\bibitem[Sha80]{Shah}
J.~Shah, 
A complete moduli space for K3 surfaces of degree $2$, 
Ann.\ of Math.\ (2) 112 (1980), no. 3, 485--510. 

\bibitem[Sha81a]{Shah2}
J.~Shah, 
Degenerations of $K3$ surfaces of degree $4$, 
Trans.\ Amer.\ Math.\ Soc.\ 263 (1981), no. 2, 271--308. 

\bibitem[Sha81b]{Shah.Enriques}
J.~Shah, 
Projective degenerations of Enriques' surfaces, 
Math.\ Ann.\ 256, 475-495 (1981). 

\bibitem[She83]{She}
N.ShepherdBarron, 
Extending polarizations on families of K3 surfaces, 
(Cambridge, MA, 1981), 
Progr. Math vol 29, 135-171, 
Birkh\"auser, Boston, 1983. 

\bibitem[Ste88]{Sterk.thesis}
H.Sterk, 
Compactifications of the period spaces of Enriques surfaces, 
- arithmetic and geometric aspects, Ph.D thesis, Katholieke Unitersiteit Nijmegen, Netherland, 1988. 

\bibitem[Ste91]{Sterk.I}
H.Sterk, 
Compactifications of the period space of Enriques surfaces, 
Part I, Math. Z., 207, 1-36 (1991), 

\bibitem[Ste95]{Sterk.II}
H.Sterk, 
Compactifications of the period space of Enriques surfaces, 
Part II, Math. Z., 427-444 (1995). 

\bibitem[Sva19]{Sva}
R. Svaldi, Hyperbolicity for log canonical pairs and the cone theorem, Selecta Math. vol.\ 25 (2019), no. 5. 


\bibitem[Sym03]{Sym}
M. Symington, 
Four dimensions from two in symplectic topology, 
Topology and geometry of manifolds (Athens, GA, 2001), Proc. Sympos. 
Pure Math., vol.\ 71., A.M.S., Provicend, RI, 2003. 

\bibitem[Tod79]{Tod.surj}
A.~Todorov, 
The period mapping that is surjective for $K3$-surfaces 
representable as a double plane, 
Mat.\ Zametki 26 (1979), no.~3, 465--474, 494. 


\bibitem[Vas99]{Vasiu}
A. Vasiu, Integral canonical models of Shimura varieties of preabelian type, Asian J. Math. 3 (1999), no. 2, 401-518.



\end{thebibliography}
\end{document}